\documentclass[a4paper,10pt,fleqn,centertags]{amsart}
\usepackage[latin1]{inputenc}
\usepackage[english]{babel}
\usepackage{cite}
\usepackage{verbatim}
\usepackage{color}
\usepackage{pgf,tikz}
\usetikzlibrary{arrows}
\usepackage{graphicx,picture}
\usepackage{amsmath,amsthm}
\usepackage{comment}
\usepackage[active]{srcltx}
\usepackage{hyperref}
\usepackage{stmaryrd}
\usepackage[top=1in, bottom=1in, left=1in,
right=1in]{geometry}

\newtheorem{theo}{Theorem}[section]

\newtheorem{prop}{Proposition}[section]

\theoremstyle{definition}

\newtheorem{rem}{Remark}[section]

\numberwithin{equation}{section}

\newcommand{\R}{\mathbb R}
\newcommand{\W}{\mathcal W}
\newcommand{\de}{\partial}
\newcommand{\eps}{\varepsilon}

\DeclareMathOperator{\divergenza}{div}

\DeclareMathOperator{\sign}{sign}

\begin{document}
\title[Sharp a priori estimates]{Sharp estimates and existence for anisotropic elliptic problems with general growth in the gradient} 
\author[F. Della Pietra, N. Gavitone]{Francesco Della Pietra and
  Nunzia Gavitone}
\address{Francesco Della Pietra \\
Universit\`a degli studi di Napoli Federico II\\
Dipartimento di Matematica e Applicazioni ``R. Caccioppoli''\\
80126 Napoli, Italia.
}
\email{f.dellapietra@unina.it}

\address{Nunzia Gavitone,
Universit\`a degli studi di Napoli Federico II \\
Dipartimento di Matematica e Applicazioni ``R. Caccioppoli''\\
80126 Napoli, Italia.
}
\email{nunzia.gavitone@unina.it}
\keywords{Nonlinear elliptic problems with gradient dependent terms,
  anisotropic Laplacian, convex symmetrization, a priori estimates}
\subjclass[2010]{35J60 35B45}
\date{\today}
\maketitle
\begin{abstract}
  In this paper, we prove sharp estimates and existence results for anisotropic nonlinear elliptic problems with lower order terms depending on the gradient. Our prototype is: 
  \begin{equation*}
\left\{
\begin{array}{ll}
-\mathcal Q_{p}u =[H(Du)]^{q}+f(x) &\text{in }\Omega,\\
u=0&\text{on }\de\Omega.
\end{array}
\right.
\end{equation*}
Here $\Omega$ is a bounded open set of $\R^{N}$, $N\ge 2$, $0<p-1<q\le p<N$, and $\mathcal Q_{p}$ is the anisotropic operator
\[
\mathcal Q_{p} u =\divergenza \left( [H(Du)]^{p-1}H_{\xi}(Du) \right),
\] 
where $H$ is a suitable norm of $\R^{N}$. Moreover, $f$ belongs to an appropriate Marcinkiewicz space. 
\end{abstract}

\section{Introduction}
Let $\Omega$ be a bounded open set of $\R^{N}$, $N\ge 2$, and $1<p<N$. Consider a convex, 1-homogeneous function $H\colon \R^{N}\rightarrow [0,+\infty[$ in $ C^{1}(\R^{N}\setminus\{0\})$. The aim of this paper is to obtain sharp a priori estimates and existence results for elliptic Dirichlet problems modeled on the following:
\begin{equation}
\label{intro:pb}
\left\{
\begin{array}{ll}
-\mathcal Q_{p}u =[H(Du)]^{q}+f(x) &\text{in }\Omega,\\
u=0&\text{on }\de\Omega,
\end{array}
\right.
\end{equation}
where $p-1<q\le p$, and $Q_{p}$ is the anisotropic operator
\[
\mathcal Q_{p} u =\divergenza \left( [H(Du)]^{p-1}H_{\xi}(Du) \right).
\] 
Moreover, we assume that $f$ belongs to the Marcinkiewicz space $M^{\frac N\gamma}(\Omega)$, with $\gamma=\frac{q}{q-(p-1)}$. In order to consider a datum $f$ which is (at least) in $L^{1}$, we will suppose that $\frac{N}{N-1}(p-1)<q\le p$. In general, $\mathcal Q_{p} $ is highly nonlinear, and it extends some well-known classes of operators. In particular, for $H(\xi)= ( \sum_k |\xi_k|^r )^\frac{1}{r}$, $r>1$, $\mathcal Q_{p}$ becomes
\begin{equation*}
  \mathcal Q_{p} v= \sum_{i=1}^{N} \frac{\de}{\de x_i} \left(
    \left( \sum_{k=1}^{N} \left| \frac{\de v}{\de x_k} \right|^r
    \right)^{(p-r)/r} \left| \frac{\de v}{\de x_i}
    \right|^{r-2}\frac{\de v}{ \de x_i}\right).
\end{equation*}
Note that for $r=2$, it coincides with the usual $p$-Laplace
operator, while for $r=p$ it is the so-called pseudo-$p$-Laplace
operator. This kind of operators has attracted an increasing interest in recent years. We refer, for example, to \cite{dpg3,ferkaw,aflt} ($p=2$) and \cite{dpgtors,dpg4,bkj06,bfk} ($1<p<+\infty$) where Dirichlet boundary conditions are considered. Moreover, for Neumann boundary values see for instance \cite{dpg2,wxpac} ($p=2$), while for the Robin case see \cite{dpgrobin}. 

 In the Euclidean setting, that is when $H(\xi)=(\sum_{i}\xi_{i}^{2})^{\frac12}$, problem \eqref{intro:pb} reduces to
\begin{equation}
\label{intro:euc}
\left\{
\begin{array}{ll}
-\Delta_{p}u =|Du|^{q}+f(x) &\text{in }\Omega,\\
u=0&\text{on }\de\Omega,
\end{array}
\right.
\end{equation}
where $\Delta_{p}$ is the well-known $p$-Laplace operator. 

Problem \eqref{intro:euc} has been widely studied in literature. In general, for equations with $q$-growth in the gradient, existence results can be given under suitable sign conditions on the gradient-dependent term (see for example \cite{bmp92} and the references therein). On the other hand, if $f\in L^{r}(\Omega)$, in order to obtain an existence result for \eqref{intro:euc} it is necessary to impose a smallness assumption on the $L^{r}$ norm of $f$. For example, if $f\in L^{r}$, $r>\frac N p$, and $\|f\|_{r}$ is small enough, then a bounded solution exists (see for instance \cite{kazkra89,mps}). As regards the case of unbounded solutions, depending on the summability of $f$, several results are known. 
For example, in \cite{fm00}, the case of $q=p$ and $f\in L^{N/p}$ is considered, and a sharp condition (in a suitable sense) on $\|f\|_{N/p}$ is given. For the general case $p-1<q\le p$, with different summability assumptions of $f$, we refer the reader to\cite{afm14,abddaper,dp,femu13,dpg4,dpper,gmp12,fermes,hmv99,tr03,chi00}. 

In this paper we deal with a problem whose prototype is \eqref{intro:pb}, for a general norm $H$ (see Section 2 for the precise assumptions), and looking for solutions in $W_{0}^{1,q}(\Omega)$ not necessarily bounded. More precisely, under a suitable smallness hypothesis on $\|f\|_{M^{N/\gamma}}$, $\gamma=\big(\frac{q}{p-1}\big)'$, we obtain some sharp a priori estimates, comparing the solutions of suitable approximating problems of \eqref{intro:pb}, with the solutions of the anisotropic radially symmetric problem 
\begin{equation}
\label{intro:pbrad}
\left\{
\begin{array}{ll}
-\mathcal Q_{p}u =[H(Du)]^{q}+\frac{\lambda}{H^{o}(x)^{\gamma}} &\text{in }\Omega^{\star},\\
u=0&\text{on }\de\Omega^{\star}.
\end{array}
\right.
\end{equation}
Here $H^{o}$ is the polar function of $H$, $\Omega^{\star}$ is the sublevel set of $H^{o}$ with the same Lebesgue measure of $\Omega$ and $\lambda=\kappa_{N}^{\gamma/N}\|f\|_{M^{{\frac N \gamma}}}$, with $\kappa_{N}=|\{x\colon H^{o}(x)< 1\}|$ (see  Section 2 for the precise definitions).  
The comparison result is obtained by means of symmetrization techniques. Taking into account the structure of the equation, we use a suitable notion of symmetrization, known as convex symmetrization (see \cite{aflt}, and Section 2 for the definition). In this order of ideas, to obtain uniform bounds on the solutions of approximating problems it is sufficient to study the anisotropic radial problem \eqref{intro:pbrad}. Hence, a key role is played by an existence and uniqueness result for a special class of positive solutions of \eqref{intro:pbrad} whose level sets are homothetic to $H^{o}$. This kind of solutions $u$ are exactly the ones that allow to perform a change of variable 
$V=\varphi(u)$, such that $V$ solves
\begin{equation}
\label{intro:pbradcambio}
\left\{
\begin{array}{ll}
-\mathcal Q_{p}V = \frac{\lambda}{H^{o}(x)^{\gamma}}\left(\frac{V+1}{\gamma-1}\right)^{\gamma-1} &\text{in }\Omega^{\star},\\
V=0&\text{on }\de\Omega^{\star}.
\end{array}
\right.
\end{equation}
The solutions of \eqref{intro:pbradcambio} can be explicitly written, and then also the solutions of \eqref{intro:pbrad}. 

The structure of the paper is the following. In Section 2, we recall the notation and the main assumptions used throughout all the paper, and we state the main results. In Section 3, we study the anisotropic radial problem \eqref{intro:pbrad}. Finally, in Section 4 we prove the quoted comparison result and a priori estimates for the approximating problems. Finally, we give the proof of the main results.

\section{Notation, preliminaries and main results}
Let $N\ge 2$, and $H:\R^N\rightarrow [0,+\infty[$ be a $C^1(\mathbb
R^N\setminus\{0\})$ function such that 
\begin{equation}
  \label{eq:omo}
  H(t\xi)= |t| H(\xi), \quad \forall \xi \in \R^N,\; \forall t \in \R,
\end{equation}
and such that any level set $\{\xi\in \R^{n}\colon H(\xi)\le t\}$, 
with $t>0$ is strictly convex. 
Moreover, suppose that there exist two positive constants $c_1
\le c_2$ such that
\begin{equation}
  \label{eq:lin}
  c_1|\xi| \le H(\xi) \le c_2|\xi|,\quad \forall \xi\in \R^N.
\end{equation}

\begin{rem}
	We stress that the homogeneity of $H$ and the convexity of its level sets imply the convexity of $H$. Indeed, by \eqref{eq:omo}, it is sufficient to show that, for any $\xi_{1},\xi_{2}\in \R^{n}\setminus\{0\}$,
	\begin{equation}
	\label{tesirem}
		H(\xi_{1}+\xi_{2}) \le H(\xi_{1}) + H(\xi_{2}).
	\end{equation}
	By the convexity of the level sets, we have
	\begin{multline*}
		H\!\left(\frac{\xi_{1}}{H(\xi_{1})  + H(\xi_{2})} + 
	\frac{\xi_{2}}{H(\xi_{1}) \! +\! H(\xi_{2})}
	\right)=\\ =
		H\left(\frac{H(\xi_{1})}{H(\xi_{1}) + H(\xi_{2})} \frac{\xi_{1}}{H(\xi_{1})} + 
	\frac{H(\xi_{2})}{H(\xi_{1})+ H(\xi_{2})} \frac{\xi_{2}}{H(\xi_{2})}
	\right) \le 1,
	\end{multline*}
	and by \eqref{eq:omo} we get \eqref{tesirem}.
\end{rem}

We define the polar function $H^o\colon \R^N\rightarrow [0,+\infty[$
of $H$ as
\[
H^o(v)=\sup_{\xi \ne 0} \frac{\xi\cdot v}{H(\xi)}.
\]
It is easy to verify that also $H^o$ is a convex function
which satisfies properties \eqref{eq:omo} and
\eqref{eq:lin}. Furthermore,
\[
H(v)=\sup_{\xi \ne 0} \frac{\xi \cdot v}{H^o(\xi)}.
\]
The set
\[
\mathcal W = \{  \xi \in \R^N \colon H^o(\xi)< 1\}.
\]
is the so-called Wulff shape centered at the origin. We put
$\kappa_N=|\mathcal W|$, and denote $\mathcal W_r=r\mathcal W$. 

In the following, we often make use of some well-known properties of
$H$ and $H^o$:  
\begin{gather*}
  H(\xi)=DH(\xi)\cdot \xi,\; H^o(\xi)=DH^o(\xi)\cdot \xi,
  \quad \forall \xi \in \R^N\setminus \{0\},
  \\
  H(D H^o(\xi))=H^o(D H(\xi))=1,\quad \forall \xi \in
  \R^N\setminus \{0\},
  \\
  H^o(\xi) D H(D H^o(\xi) ) = H(\xi) D H^o(D H(\xi) ) = \xi,\quad
  \forall \xi \in \R^N\setminus \{0\}.
\end{gather*}

Let $\Omega$ be an open subset of $\mathbb R^N$. The total variation
of a function $u\in BV(\Omega)$ with respect to $H$ is (see \cite{ab}):
\[
\int_\Omega |Du|_H = \sup
\left\{
  \int_\Omega u\divergenza \sigma dx\colon
  \sigma \in C_0^1(\Omega;\R^N),\; H^o(\sigma)\le 1
\right\}.
\]
This yields the following definition of anisotropic perimeter of
$F\subset \R^N$ in $\Omega$:
\[
P_H(F;\Omega) = \int_\Omega |D\chi_F|_H= \sup
\left\{
  \int_F \divergenza \sigma dx\colon \sigma \in C_0^1(\Omega;\R^N),\;
  H^o(\sigma)\le 1
\right\}.
\]
The following co-area formula for the anisotropic perimeter
\begin{equation}\label{fr}
  \int_{\{u>t\}} H(Du) dx =
  \int_\Omega P_H (\{u>s\},\Omega)\, ds,\quad
  \forall u\in BV(\Omega)
\end{equation}
holds, moreover
\[
P_H(F;\Omega)= \int_{\Omega\cap \partial^*F} H(\nu_F) d\mathcal H^{N-1}
\]
where $\mathcal H^{N-1}$ is the $(N-1)-$dimensional Hausdorff measure
in $\mathcal R^N$, $\partial^*F$ is the reduced boundary of $F$ and
$\nu_F$ is the outer normal to $F$ (see \cite{ab}). 

The anisotropic perimeter of a set $F$ is  finite if and only if the
usual Euclidean perimeter 
\[
P(F;\Omega)=  \sup
\left\{
  \int_F \divergenza \sigma dx \colon
  \sigma \in C_0^1(\Omega;\R^N),\; |\sigma|\le 1
\right\}.
\]
is finite. Indeed, by properties \eqref{eq:omo} and \eqref{eq:lin} we
have that
\begin{equation}
\label{eq:lin2}
  \frac{1}{c_2} |\xi| \le H^o(\xi) \le \frac{1}{c_1} |\xi|,
\end{equation}
and then
\begin{equation*}
  c_1 P(E;\Omega) \le P_H(E;\Omega) \le c_2 P(E;\Omega).
\end{equation*}
A fundamental inequality for the anisotropic perimeter is the
isoperimetric inequality
\begin{equation}
  \label{isop}
  P_H(E;\R^N) \ge N \kappa_N^{\frac 1 N} |E|^{1-\frac 1 N},
\end{equation}
which holds for any measurable subset $E$ of $\R^N$ (see \cite{bu,
dpf,fomu,aflt}. See also \cite{dpg1} for some questions related to an anisotropic relative isoperimetric inequality).

We recall that if $u\in W^{1,1}(\Omega)$, then (see \cite{ab})
\[
\int_{\Omega} |Du|_H =\int_\Omega H(Du) dx.
\]

\subsection{Rearrangements and convex symmetrization}
We recall some basic definition on rearrangements.
Let $\Omega$ be an bounded open set of $\R^N$, 
$u:\Omega\rightarrow\R$ be a measurable function, and denote with
$|\Omega|$ the Lebesgue measure of $\Omega$.

The {distribution function} of $u$ is the map
$\mu_u:\mathbb R \rightarrow[0,\infty[$ defined by
\begin{equation*}
  \mu_u(t)\,=\,|\{x\in\Omega:|u(x)|>t\}|.
\end{equation*}
Such function is decreasing and right continuous.

The {decreasing rearrangement} of $u$ is the map
$u^*:\,[0,\infty[\rightarrow \R$ defined by
\begin{equation*}
  u^*(s):=\sup\{t\in\R:\mu_u(t)>s\}.
\end{equation*}
The function $u^*$ is the generalized inverse of $\mu_u$.

Following \cite{aflt}, the convex symmetrization of $u$ is the
function $u^\star(x)$, $x\in \Omega^\star$ defined by:
\begin{equation*}
  u^\star(x)=u^*(\kappa_N H^o(x)^N),
\end{equation*}
where $\Omega^\star$ is a set homothetic to the Wulff shape centered
at the origin having the same measure of $\Omega$, that is,
$\Omega^\star=\mathcal W_R$, with $R=\big(\frac{|\Omega|}
{\kappa_N}\big)^{1/N}$. 
  
We will say that any $w(x)$, $x\in \Omega^{\star}$ is an anisotropic radial function 
if for any $x\in \Omega^{\star}$, $w(x)=\tilde w(H^{o}(x))$, 
for some function $\tilde w(r)$, $r\in [0,R]$. For the sake of brevity, we will refer to such functions as radial functions. For example, $u^{\star}$ is radial.

The following results will be useful in the sequel. First, a basic
tool will be the Hardy inequality, stated below.
\begin{prop} For any $u\in W^{1,\gamma}(\R^N)$, $1<\gamma<N$,
  \begin{equation}
    \label{Hardy}
    \int_{\R^N} H(Du)^\gamma dx \ge \Lambda_\gamma \int_{\R^N}
    \frac{|u|^\gamma}{H^o(x)^\gamma} dx,
  \end{equation}
  and the constant $\Lambda_\gamma = \left(\frac{N-\gamma}{\gamma}\right)^\gamma$ is
  optimal, and it is not achieved.
\end{prop}
If $H(x)=\vert x\vert$, \eqref{Hardy} is the classical Hardy
inequality. For a general $H$, \eqref{Hardy} is proved in \cite{vs}.

Finally, we recall the definition of Marcinkiewicz spaces.  We say that a measurable function $u\colon \Omega\rightarrow \mathbb R$ belongs to $M^{r}(\Omega)$, $r>1$, if there exists a constant $C$ such that
\[
	\mu_{u}(t) \le {C}{t^{-r}},\quad\forall t>0,
\]
or, equivalently,
\[
u^{*}(s) \le {C}{\sigma^{-\frac{1}{r}}},	\quad\forall \sigma\in]0,|\Omega|]. 
\]
Then, we denote
\[
\|u\|_{M^{r}(\Omega)}=\sup_{\sigma \in ]0,|\Omega|[} u^{*}(\sigma)\sigma^{\frac 1 r}.
\] 
\subsection{Statement of the problem and main results}
Our aim is to prove a priori estimates and existence results for
problems of the type
\begin{equation}
  \label{eq:gen}
  \left\{
    \begin{array}{ll}
      -\divergenza{\left(a(x,u,Du)\right)}=b(x,u,Du)+ f(x) &
      \text{in }\Omega, \\
      u = 0 & \text{on }\de\Omega,
    \end{array}
  \right.
\end{equation}
where $a\colon \Omega\times\R\times \mathbb R^N\rightarrow \mathbb
R^N$ is a Carath\'eodory functions verifying 
\begin{equation}
  \label{ellipt}
  a(x,s,\xi) \cdot \xi \ge H(\xi)^p,
\end{equation}
and
\begin{equation}
  \label{growth}
  |a(x,s,\xi)|\le \alpha (|\xi|^{p-1}+|s|^{p-1}+k(x)), 
\end{equation}
for a.e. $x\in\Omega$, for any  $(s,\xi)\in\R\times\R^N$, where
$\alpha>0$, $k\in L_+^{p'}(\Omega)$, and $1<p<N$. 
Moreover,
\begin{equation}
  \label{ip-mon}
  (a(x,s,\xi) - a(x,s,\xi')) \cdot (\xi-\xi')>0,
\end{equation}
for a.e. $x\in\Omega$, for all $s\in \R,\xi\neq\xi'\in \R^N$.
As regards the lower order terms, we suppose that $b\colon
\Omega\times\R\times \mathbb R^N\rightarrow \mathbb R$ is a
Carath\'eodory functions such that
\begin{equation}
  \label{b=1}
  |b(x,s,\xi)|\le H(\xi)^q
\end{equation}
for a.e. $x\in\Omega$, for any  $(s,\xi)\in\R\times\R^N$, for
$p-1<q\le p$.

Finally, we take $f$ such that 
  \begin{equation}
  \label{ipf}
  f^\star(x) \le \frac{\lambda}{H^o(x)^\gamma},\quad x\in
  \Omega^\star,\quad \text{ with }\gamma=
  \left(\frac{q}{p-1}\right)'=\frac{q}{q-(p-1)}. 
  \end{equation}
  
 We observe that such hypothesis implies that $f$
 belongs to the Marcinkiewicz space $M^{\frac{N}{\gamma}}(\Omega)$. It is
 worth to recall that $M^{\frac{N}{\gamma}}(\Omega)\subset L^{s}(\Omega)$ for
 any $s<\frac{N}{\gamma}$, but $M^{\frac{N}{\gamma}}(\Omega)\supset
 L^{\frac{N}{\gamma}}(\Omega)$.

 Assume first that 
 \begin{equation}
 	\label{q1}
	p\ge q>p-1+\frac{p}{N}.
 \end{equation} Then \eqref{ipf} implies that $f\in L^{(p^*)'}(\Omega)$, where $p^*=\frac{Np}{N-p}$
  is the Sobolev conjugate of $p$. Indeed in this case, ${p} \le \gamma < \frac{Np-N+p}{p}$,
that is $\frac N \gamma> (p^*)'$.
Hence, if \eqref{q1} holds, we say that $u\in W^{1,p}_0(\Omega)$
 is a weak solution of \eqref{eq:gen} if
 \begin{equation}\label{defsol}
  \int_\Omega a(x,u,Du)\cdot D\varphi\, dx = \int_\Omega [b(x,u,Du)
  +f]\varphi\, dx, 
\end{equation}
for any $\varphi\in W^{1,p}_0(\Omega)\cap L^{\infty}(\Omega)$. 

Second, suppose that
\begin{equation}
 	\label{q2}
 p-1+\frac{p}{N}\ge q > \frac{N}{N-1}(p-1).
\end{equation}
Then \eqref{ipf} gives that $f\in L^s(\Omega)$, with $1<s<\frac N \gamma$. Hence, if \eqref{q2} holds,
we say that $u$ is a distributional solution of \eqref{eq:gen} if $u\in
W^{1,q}_0(\Omega)$ and \eqref{defsol} is satisfied for any $\varphi\in C^{\infty}_0(\Omega)$.

Finally, let us observe that if $q=\frac{N}{N-1}(p-1)$,
then $\frac N \gamma=1$, and $f$ is not in $L^1(\Omega)$. 

The main results of our paper will be the following.
\begin{theo}
  \label{thm:ex}
Suppose that the assumptions \eqref{ellipt}--\eqref{b=1} hold. Moreover, let
  $f\in M^{\frac N \gamma}(\Omega)$ such that  
  \begin{equation}
  	\label{condf}
  f^\star(x) \le \frac{\lambda}{H^o(x)^\gamma},\quad x\in
  \Omega^\star,\text{ for some }0\le \lambda< c_\gamma\Lambda_\gamma,
  \end{equation}
  with $c_\gamma=(\gamma-1)^{\gamma-1}$, and
  $\Lambda_\gamma=\left(\frac{N-\gamma}{\gamma}\right)^\gamma$.
  Then,
  \begin{itemize}
  \item[(a)] if $p\ge q >p-1+\frac p N$, then problem
    \eqref{eq:gen} admits a weak solution $u\in W_0^{1,p}(\Omega)\cap
    L^s(\Omega)$, with $s<+\infty$ if $p=q$, or
    $s<\frac{N[q-(p-1)]}{p-q}$ otherwise. 
  \item[(b)] if $p-1+\frac p N \ge q >\frac{N}{N-1}(p-1)$, then problem
    \eqref{eq:gen} admits a distributional solution $u\in
    W_0^{1,r}(\Omega)$, with $r<N[q-(p-1)]$. 
  \end{itemize}
  \end{theo}
 \begin{theo}
 \label{thm:comp}
  Under the assumptions of Theorem \ref{thm:ex}, the obtained solution $u$ 
  verifies 
  \[
  u^*(s) \le v^*(s),\quad s\in]0,|\Omega|],
  \]
  where $v\in W_{0}^{1,s}(\Omega^{\star})$, $s<N(q-(p-1))$ is the radial solution of problem
  \[
  \left\{
    \begin{array}{ll}
      -\mathcal Q_p v= H(Dv)^q + \dfrac{\lambda}{H^o(x)^\gamma} &
      \text{in }\Omega^{\star}, \\[.2cm]  
      v = 0 & \text{on }\de \Omega^{\star},
    \end{array}
  \right.
  \]
  given in Theorem \ref{unic2} (See Section 3).
\end{theo}

\begin{rem}
 We explicitly observe that requiring the condition \eqref{condf} on $f$ in the above Theorems is equivalent to assume that
 \[
 \|f\|_{M^{\frac N \gamma}} < \kappa_{N}^{\frac\gamma N}(\gamma-1)^{\gamma-1}\left( \frac{N-\gamma}{\gamma}\right)^{\gamma}.
 \]
\end{rem}

\section{The radial case}
We first study the problem
\begin{equation}\label{eq:rad1}
  \left\{
    \begin{array}{ll}
      -\mathcal Q_p v= H(Dv)^q + \dfrac{\lambda}{H^o(x)^\gamma} &
      \text{in }\W_R, \\[.2cm]  
      v = 0 & \text{on }\de \W_R,
    \end{array}
  \right.
\end{equation}
where $\lambda\ge 0$, $p-1<q\le p$, $\mathcal W_R$ is
the Wulff shape centered at the origin and radius $R$, and $\gamma =
\left(\frac{q}{p-1}\right)'=\frac{q}{q-(p-1)}$. 
 
In order to prove an existence and uniqueness result for problem \eqref{eq:rad1},
we first study the following problem:
\begin{equation}\label{eq:rad1cv}
  \left\{
    \begin{array}{ll}
      -\mathcal Q_\gamma V= \dfrac{\lambda
        }{c_\gamma H^o(x)^\gamma}(V+1)^{\gamma-1} & 
    \text{in } \W_R, \\[.2cm]  
      V= 0 & \text{on }\de \W_R,
    \end{array}
  \right.
\end{equation}
with $c_\gamma =  (\gamma - 1)^{\gamma-1}$. 
\begin{rem}
\label{rembeta}
If we look for radial solutions $V(r)=V(H^{o}(x))$ 
of \eqref{eq:rad1cv}, these solves the equation 
  \begin{equation}\label{eq1dim}
  -|V'|^{\gamma-2}\left( (\gamma-1)V''+\dfrac{N-1}{r} V' \right)=
  \frac{\lambda}{c_\gamma}\frac{(V+1)^{\gamma-1}}{r^\gamma} \quad
  \text{in }]0,R[,
  \end{equation}
 which follows from the equation in \eqref{eq:rad1cv}, plugging in the function $V(r)=V(H^o(x))$ and using the properties of $H$. It is a straightforward computation to show that
  $\Phi (r)=\left(\frac{R}{r}\right)^{\beta}-1$ solves \eqref{eq1dim} if
  and only if $\beta$ is such that  
  	\begin{equation}
	\label{betaeq}
		-(\gamma-1)\, \beta^\gamma + (N-\gamma)\, \beta^{\gamma-1}=\frac{\lambda}{c_\gamma}.
	\end{equation}
  For $0\le \lambda < c_\gamma  \Lambda_{\gamma}$, this equation has exactly
  two different solutions, but there exists a unique solution
  $\beta$ such that 
  \begin{equation}
  \label{betagiusto}
  \beta\in\left[0,\frac{N-\gamma}{\gamma}\right[\quad\text{and}\quad 
  	\Phi(x)=\left(\frac{R}{H^{o}(x)}\right)^{\beta}-1 \in W_0^{1,\gamma}(\W_R)
  \end{equation}    
  (see Figure \ref{fig1}).
\end{rem}
  \begin{figure}[h]
\definecolor{ttqqtt}{rgb}{0.2,0,0.2}
\begin{tikzpicture}[line cap=round,line join=round,>=triangle
  45,x=1.0cm,y=1.0cm,scale=10,>=stealth] 
\draw[->,color=black] (-0.04,0) -- (1.11,0);
\draw[->,color=black] (0,-0.05) -- (0,0.38);
\foreach \y in {0.297}
\draw[shift={(0,\y)},color=black] (.2pt,0pt) -- (-.2pt,0pt);
\draw (0,.297) node [anchor=east]
{\scriptsize$\Lambda_\gamma$};
\draw[dotted] (0,.297) -- (.67,.297);
\clip(-0.1,-0.07) rectangle (1.12,0.381);
\draw[smooth,samples=200,domain=0.0:1.0] plot(\x,{0-2*(\x)^3+2*(\x)^2});
\draw (.67,-.2pt) -- (.67,.2pt) node [anchor=north] {\scriptsize
  $\frac{N-\gamma}{\gamma}$};
\draw[dotted] (.67,.292) -- (.67,0);
\draw (1,0) node [anchor=north] {\scriptsize
  $\frac{N-\gamma}{\gamma-1}$}; 
\draw (1.11,0) node [anchor=south] {\scriptsize $\beta$};
\draw (0,0.36) node [anchor=west] {\scriptsize $F(\beta)$};
\draw[dotted] (-.2pt,.17) node [anchor=east]{\scriptsize $\frac{\lambda}
  {c_\gamma}$} -- (.89,.17);
\draw[dotted] (.366,0) node [anchor=north]{\scriptsize $\beta$}
-- (.366,.17);
\end{tikzpicture}
\caption{$F(\beta)= -(\gamma-1) \beta^\gamma + (N-\gamma)
  \beta^{\gamma-1}$. For any $ \lambda \in [0, c_\gamma
  \Lambda_\gamma[$, there exists a unique $\beta\ge 0$ such that 
  $F(\beta)=\frac{\lambda}{c_\gamma}$ and $r^{-\beta}\in
 W^{1,\gamma}(\W_R) $.}
\label{fig1}
\end{figure}
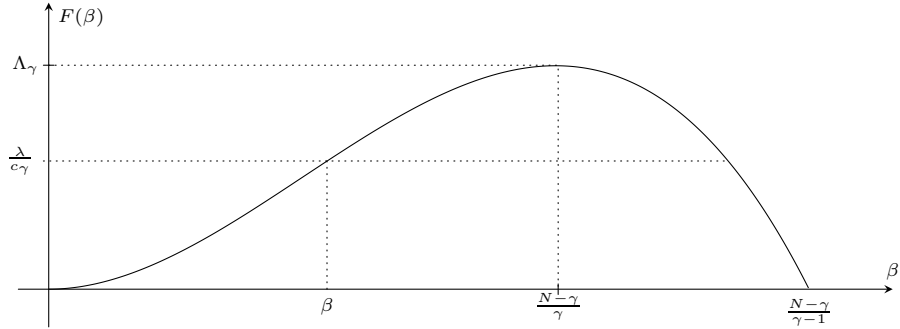

The following result holds:
\begin{theo}
  \label{unic} Let $1<\gamma<N$, and 
  \[
  0\le \lambda <c_\gamma \Lambda_\gamma,
  \]
 where $\Lambda_\gamma=\left(\frac{N-\gamma}{\gamma}\right)^{\gamma}$
 is the best constant of the Hardy inequality \eqref{Hardy}. Then, if $\lambda>0$, the 
 problem \eqref{eq:rad1cv} admits a unique positive solution $\Phi \in
 W_0^{1,\gamma}(\W_R)$, in the sense that 
 \begin{equation} \label{sol-gamma}
 \int_{\W_R} H(D\Phi)^{\gamma-1}H_\xi(D\Phi) \cdot D\varphi\, dx =
 \frac{\lambda}{c_\gamma}
   \int_{\W_R}\frac{1}{H^o(x)^{\gamma}}(\Phi +1)^{\gamma-1}\varphi\,dx,\quad
   \varphi \in W_0^{1,\gamma}(\W_R),
 \end{equation}
where $\Phi$ is given in \eqref{betagiusto}. 
 Moreover, if $\lambda=0$ the unique solution in $W_{0}^{1,\gamma}(\Omega)$ to \eqref{eq:rad1cv} in the sense of \eqref{sol-gamma} is $\Phi\equiv 0$.
   \end{theo}
\begin{proof}
By Remark \ref{rembeta}, we have to prove only the uniqueness issue. We first assume that $0<\lambda<c_{\gamma}\lambda_{\gamma}$. Reasoning as in \cite{bk02,bfk}, we prove that there are no other positive solutions in $W_{0}^{1,\gamma}(\Omega)$ of 
\eqref{eq:rad1cv}. As a matter of fact, the positive solutions of \eqref{eq:rad1cv} are stationary points of the functional
\begin{equation}
  \label{eq:funct}
  F(\psi)= \frac 1 \gamma \int_{\W_R} \left[H(D\psi)^\gamma - \frac{\lambda
    }{c_\gamma H^o(x)^\gamma} [(|\psi|+1)^{\gamma}-1]\sign \psi\right] dx,\quad \psi \in
  W_0^{1,\gamma}(\W_R).
\end{equation}
The functional $F(\psi)$ is even. Moreover, it is strictly convex in the variable $\psi^\gamma$. 
Indeed, if $U,V>0$, $U,V\in W_{0}^{1,\gamma}(\Omega)$, then 
the function
\[
\phi= \left(\frac{U ^{\gamma}+V^{\gamma}}{2}\right)^{1/\gamma}
\]
is an admissible test function for $F$ in \eqref{eq:funct}. Computing
$D\phi$, by the homogeneity of $H$ it follows that
\[
H(D\phi)= \phi\, H\left( \frac 1 2 \frac{U ^\gamma}{\phi^\gamma}
  \frac{DU }{U } + \frac 1 2 \frac{V^\gamma}{\phi^\gamma}
  \frac{DV}{V} \right).
\]
Let $s(x)=\frac{Z^\gamma}{2\phi^\gamma}$. Observing that $0<s<1$, by
convexity and homogenity of $H$ we have that
\begin{equation*}
\begin{array}{rl}
H(D\phi)^\gamma =& \phi^\gamma H\left(s(x) \dfrac{D U }{U }+
  (1-s(x))\dfrac{DV}{V} \right)^\gamma \\
\le & \phi^\gamma \left(s(x)H\left(\dfrac{DU }{U }\right)^\gamma
  +(1-s(x))H\left( \dfrac{DV}{V} \right)^\gamma\right) \\
=& \dfrac{U ^\gamma}{2} H\left(\dfrac{D U }{U }\right)^\gamma
 +\dfrac{V^\gamma}{2} H\left(\dfrac{DV}{V}\right)^\gamma\\[.3cm]
 =& \dfrac 1 2 \left[H(DU )^\gamma + H(DV)^\gamma\right].
\end{array}
\end{equation*}
On the other hand, the function $g(t)=(t^{1/\gamma}+1)^\gamma$, $t\ge 0$ is strictly concave, and then $F(\psi)$ is stricly convex in $\psi^{\gamma}$. Finally, $F$ admits only the positive critical point $\Phi$.

The theorem is completely proved if we show that, when $\lambda=0$, $\Phi=0$ is the unique solution in $W_{0}^{1,\gamma}$. This follows observing that, in this case, the functional $F$ becomes
\[
F(\psi)=\frac 1 \gamma \int_{\mathcal W_{R}}[H(D\psi)^{\gamma}]dx,
\] 
which is strictly convex, since $H^{\gamma}(\xi)$ is strictly convex in $\xi$.
\end{proof}

\begin{rem}
It is worth noting that the argument of Theorem \ref{unic} can be used, for example, 
also in order to obtain uniqueness for problems of the type
\begin{equation}\label{modradg}
\left\{
\begin{array}{ll}
-\mathcal Q_\gamma v= b(x) |v|^{\gamma-2}v + f(x) &\text{in }
\Omega,\\[.1cm]
v=0 &\text {on }\de \Omega,\\
\end{array}
\right.
\end{equation}
with $\Omega$ bounded open set of $\R^N$, $b$ such that 
\begin{equation}\label{eq:unigen}
b(x)\in L\left(\frac N
  \gamma,\infty\right),\text { with } (b^+)^\star(x)\le
\frac{\lambda}{H^o(x)^\gamma}\text { in }\Omega^\star,\;
0<\lambda<\Lambda_\gamma, 
\end{equation}
and $f\in L((\gamma^*)',\gamma')$, $f\ge 0$, $f\not\equiv 0$
in $\Omega$. Under this assumptions, problem \eqref{modradg} admits at
most a (positive) weak solution. Indeed, if $v$ is a solution to \eqref{modradg}, using the Polya-Szeg\"o inequality in the anisotropic case (see \cite{aflt}), and the Hardy-Littlewood inequality we get that
\[
\int_{\Omega^\star} H(D(v^-)^\star)^\gamma dx \le \int_\Omega H(Dv^-)^\gamma dx \le
 \int_\Omega b^+(v^-)^\gamma dx \le \int_{\Omega^\star}
(b^+)^\star[(v^-)^\star]^\gamma dx,
\]
Recalling the assumptions on $b$ in \eqref{eq:unigen}, the Hardy
inequality assures that $v^-\equiv 0$. Actually, by the maximum principle $v$ must be positive in $\Omega$. Hence we can proceed similarly as
in the proof of Theorem \ref{unic} obtaining the uniqueness of the
solution (see also \cite{diazsaa,bk02,dpg3}).
\end{rem}

\begin{theo}\label{unic2}
  Let $p\ge q>(p-1)\frac{N}{N-1}$, $\Lambda_\gamma=\left(\frac{N-\gamma}{\gamma}\right)^\gamma$, $c_\gamma=(\gamma-1)^{\gamma-1}$, $\gamma=\left(\frac{q}{p-1}\right)'$ and
  \begin{equation*}
    0\le \lambda < c_\gamma{\Lambda_\gamma}.
  \end{equation*}
 Then,  if $\lambda>0$ there exists a unique positive, radially decreasing,
  distributional solution  $v(x)=v(r)$ of \eqref{eq:rad1} in
  $W_0^{1,s}(\W_R)$, with $s<N(q-(p-1))=\tilde s$, such that, defining
  \begin{equation}
       \label{eq:expip}
       V(x)=\exp\left[ \frac{1}{\gamma-1}\int_{H^o(x)}^{R}
         (-v'(\tau))^{q-(p-1)} d\tau \right] - 1,
     \end{equation}
     it holds that
     \begin{equation}
       \label{eq:expip1}
      V\in W_0^{1,\gamma}(\W_R),\quad (V+1)^{\gamma-1}\in
      W^{1,\delta}(\W_R),\;
      \text{for some }\delta > \tilde\delta =
      \left(
        \frac{\tilde s}{p-1}
      \right)'.
    \end{equation}
     Moreover, if $q<p$,
     \[
     v(r)= \theta
     \left[
       r^{-\frac{p-q}{q-(p-1)}} -R^{-\frac{p-q}{q-(p-1)}}
     \right],
     \]
     with $\theta=[(\gamma-1)\beta]^{\frac{1}{q-(p-1)}}
     \frac{q-(p-1)}{p-q}$, while, for $q=p$, 
     \[
     v(r)=(p-1)\beta \log \frac{R}{r},
     \]
     where $\beta$ is the solution of \eqref{betaeq} given in \eqref{betagiusto}.
     Finally, if $\lambda=0$ and $(p-1)\frac{N}{N-1} < q\le p$,
     the unique radially decreasing solution $v$ such that \eqref{eq:expip},\eqref{eq:expip1} holds is $v=0$. 
   \end{theo}
\begin{proof}
Using the notation of Theorem \ref{unic}, being $0<\lambda< c_\gamma
\Lambda_\gamma$, we can consider $\Phi =(R/r)^\beta-1$,
$0\le \beta<\frac{N-\gamma}{\gamma}$ as the unique positive solution
in $W_0^{1,\gamma}(\W_R)$ of \eqref{eq:rad1cv}. We reason as in
\cite{fermes}, performing the change of variable  
\[
v(r)= (\gamma-1)^{\frac{1}{q-(p-1)}} \int_r^R
\left(
  \frac{-\Phi '(s)}{\Phi (s)+1}
\right)^{\frac{1}{q-(p-1)}} ds = \theta
\left[
  r^{-\frac{p-q}{q-(p-1)}}-R^{-\frac{p-q}{q-(p-1)}}
\right],
\]
with $\theta=[(\gamma-1)\beta]^{\frac{1}{q-(p-1)}}
\frac{q-(p-1)}{p-q}$. A direct computation shows that, being
$q>\frac{N}{N-1}(p-1)$, $v$ belongs to $W_0^{1,s}(\W_R)$, for all
$s<\tilde s= N(q-(p-1))$ and it is a solution of \eqref{eq:rad1}. 
Moreover, the function $V(x)$ defined in \eqref{eq:expip} coincides
with $\Phi (x)$, and, being $0\le\beta< \frac{N-\gamma}{\gamma}$,
there exists $\delta>\tilde\delta$ such that \eqref{eq:expip1} holds. 
 
On the contrary, let us suppose that $v(x)\in W_0^{1,s}(\W_R)$ for any
$s<N(q-(p-1))$, with $v$ is a radially decreasing, and solves 
\eqref{eq:rad1}. Moreover, suppose that the function $V$ defined in
\eqref{eq:expip} verifies \eqref{eq:expip1}. 

Following the method contained in \cite[Proposition 1.8]{fm98}, we
show that $V(x)$ is a solution of \eqref{eq:rad1cv}, in the sense of
\eqref{sol-gamma}. Being $V\in
W_0^{1,\gamma}(\W_R)$, by a density argument and the Hardy inequality
\eqref{Hardy} it is sufficient to
show that
\begin{equation}
  \label{eq:10}
  -\mathcal Q_\gamma(V) =\frac{\lambda}{ c_\gamma H^o(x)^\gamma}
  (V+1)^{\gamma-1} \quad\text{in }\mathcal D'(\W_R). 
\end{equation}

Being $v\in W_0^{1,s}(\W_R)$ for any $s<\tilde s$,  the integral 
$\int_{\W_R} H^{p-1}(Dv)H_\xi (Dv)\cdot 
D\phi\, dx$ is finite as $\phi\in W_0^{1,\delta}(\W_R)$, with
$\delta>\tilde \delta$
given in \eqref{eq:expip1}. This ensures that the  
operator $T:=-\mathcal Q_p v$ belongs to $W^{-1,\delta'}$. Hence, by
\eqref{eq:expip1} the following product 
$(V+1)^{\gamma-1}T$ is well defined in $\mathcal D'$: 
\begin{multline*}
\langle (V+1)^{\gamma-1} T,\varphi \rangle := \langle 
T,\, (V+1)^{\gamma-1} \varphi \rangle =
\int_{\W_R} H(Dv)^{p-1} H_\xi (Dv)\cdot
D\left[(V+1)^{\gamma-1}\varphi\right] dx =
\\
=\int_{\W_R} H(Dv)^{p-1}H_\xi(Dv) \cdot \left[
  (V+1)^{\gamma-1} D\varphi +
  \varphi\, (\gamma-1) (V+1)^{\gamma-1}H(Dv)^{q-p} Dv  \right]
dx,\\ \forall \varphi \in C_0^{\infty}(\W_R).
\end{multline*}
We obtain that
\begin{equation}
  \label{eq:5}
  (V+1)^{\gamma-1} T = -\divergenza \left[ (V+1)^{\gamma-1}
  H(Dv)^{p-1}H_\xi(Dv) \right] +(V+1)^{\gamma-1} H(Dv)^q
\quad\text{in }\mathcal D'.
\end{equation}
Being $v$ a solution of \eqref{eq:rad1}, $-\mathcal
Q_p(v)=H(Dv)^q+\lambda H^o(x)^{-\gamma} \in L^1$. Furthermore,
$[H(Dv)^q+\lambda H^o(x)^{-\gamma}] (V+1)^{\gamma-1}\in L^1$. Indeed, recalling
\eqref{eq:expip1}, we have that  
\[
(V+1)^{\gamma-1} H(Dv)^q \le C \frac{|DV|^\gamma}{V+1} \in L^1,
\]
and $H^o(x)^{-\gamma} (V+1)^{\gamma-1}\varphi \in L^1$ by the
Hardy inequality. Hence, we can use the result of Brezis and Browder
\cite{bb79}, obtaining that, as $\varphi\in \mathcal C_0^{\infty}(\W_R)$, 
\begin{equation*}\label{eq:54}
\int_{\W_R} H^{p-1}(Dv) H_\xi (Dv)\cdot
D\left[(V+1)^{\gamma-1}\varphi\right] dx= \int_{\W_R}
\left[(V+1)^{\gamma-1} \left( H(Dv)^q+\frac{\lambda}{H^0(x)^{\gamma}}
  \right)\right] \varphi \,dx,
\end{equation*}
that is
\begin{equation}
  \label{eq:6}
  (V+1)^{\gamma-1}T= (V+1)^{\gamma-1}\left(
    H(Dv)^q+\frac{\lambda}{H^0(x)^{\gamma}} \right) \quad\text{in
  }\mathcal D'(\W_R).
\end{equation}
On the other hand, it is easy to see that 
\begin{equation}
  \label{eq:7}
- \mathcal Q_\gamma (V) = -\frac{1}{c_\gamma} \divergenza\left[
  (V+1)^{\gamma-1} H(Dv)^{p-1}H_\xi (Dv)  \right] \quad\text{in
}\mathcal D'(\W_R).
\end{equation}
Putting \eqref{eq:5}, \eqref{eq:6} and \eqref{eq:7} together, we get
that $V\in W_0^{1,\gamma}(\W_R)$ satisfies \eqref{eq:10}.
Then $V(x)=\Phi (x)$ by Theorem \ref{unic}, and this concludes the
proof. 
\end{proof}

\begin{rem}
We explicitly observe that problem \eqref{eq:rad1} admits at least 
two nonnegative solutions in $W_{0}^{1,s}(\mathcal W_{R})$, $\forall s<\tilde s$. 
for example, $\lambda=0$, $R=1$ and $N/(N-1)<q<p$, the problem
 \[
  -\mathcal{Q}_{p}(u)=[H(Du)]^{q}, \qquad u\in W_{0}^{1,q}(\mathcal W_{1})
\]
admits the radially decreasing solutions $u_{1}=0$ and 
\[
u_{2}(x)= K \left(\frac{1}{H^{o}(x)^{\frac{p-q}{q-(p-1)}}} - 1\right), \quad K=\frac{q-(p-1)}{p-q} \left(\frac{(N-1)q-(p-1)N}{q-(p-1)}\right)^{\frac{1}{q-(p-1)}}.
\]
As a matter of fact, $u_{2}\in W_{0}^{1,s}(\mathcal W_{1})$, $s<\tilde s$ but, making the change of variable \eqref{eq:expip}, the function
\[
   V(x)=\exp\left[ \frac{q-(p-1)}{p-1}\int_{r}^{1}
         (-u_{2}'(\tau))^{q-(p-1)} d\tau \right]  -1= \left(\frac{1}{r}\right)^{\frac{(N-1)q-(p-1)N}{q-(p-1)}}-1,\quad r=H^{o}(x)
\]
does not verify \eqref{eq:expip1}.
\end{rem}
For the uniqueness issue of problem \eqref{eq:gen}, we refer the reader to 
\cite{bapor,bm95} and the references therein.

\section{A priori estimates and proof of Theorems \ref{thm:ex} and \ref{thm:comp}}

The key role in order to prove Theorem \ref{thm:ex} is played by some a priori estimates, given in Theorem \ref{theo:comp} and in Proposition \ref{propest} below, for the approximating problems
\begin{equation}\label{eq:gen1appr}
  \left\{
    \begin{array}{ll}
      -\divergenza{\left(a(x,u_\eps,Du_\eps)\right)} =
      b_\eps(x,u_\eps,Du_\eps)
      + T_{1/\eps}(f(x)) &
      \text{in }\Omega, \\
      u_\eps = 0 & \text{on }\de\Omega,
    \end{array}
  \right.
\end{equation}
where $\eps>0$, 
\[
b_\eps(x,s,\xi)=\frac{b(x,s,\xi)}{1+\eps |b(x,s,\xi)|},\quad \text{for
  a.e. }x\in \Omega, \forall (s,\xi) \in \R\times \R^N,
\]
and $T_t(s)=\min\{s,\max\{-s,t\}\}$, $t>0$ is the standard
truncature function. Since $|b_\eps|\le 1/\eps$ and $f_{\eps}\in L^{\infty}(\Omega)$, the assumptions \eqref{ellipt}, \eqref{growth} and \eqref{ip-mon} allows to apply the classical results contained in \cite{ll,li69}. Then there exists a weak solution $u_{\eps}\in W_{0}^{1,p}(\Omega)$. Moreover, $u_{\eps}\in L^{\infty}(\Omega)$. 

The theorem below is in the spirit of the comparison results contained
in \cite{tal1,aflt,fermes}.
\begin{theo}\label{theo:comp}
  Let $u_\eps\in W_0^{1,p}(\Omega)\cap L^{\infty}(\Omega)$ be a weak
  solution of \eqref{eq:gen1appr}, under the assumptions \eqref{ellipt}-\eqref{b=1}, with $f\in M^{\frac N
    \gamma}(\Omega)$ such that
  \[
  f^\star(x) \le \frac{\lambda}{H^o(x)^\gamma},\quad x\in
  \Omega^\star,\quad\text{for some }{0\le \lambda< c_\gamma\Lambda_\gamma},
  \]
  with $c_\gamma=(\gamma-1)^{\gamma-1}$, and
  $\Lambda_\gamma=\left(\frac{N-\gamma}{\gamma}\right)^\gamma$.
  Then,
  \begin{equation}
    \label{eq:thesi}
  u_\eps^*(s)\le v^*(s), \quad s\in ]0,|\Omega|].
  \end{equation}
  where $v\in W^{1,s}_0(\Omega^\star)$, $\forall s<N(q-(p-1))$ is the
  solution of problem  
  \begin{equation*}
  \left\{
    \begin{array}{ll}
      -\mathcal Q_p v= H(Dv)^q + \dfrac{\lambda}{H^o(x)^\gamma} &
      \text{in }\Omega^\star, \\[.2cm]  
      v = 0 & \text{on }\de \Omega^\star,
    \end{array}
  \right.
\end{equation*}
given by Theorem \ref{unic2}. 
\end{theo}
\begin{proof}
The first step consists in proving the following differential inequality:
\begin{multline}
  \label{eq:ineq}
  (-u_\eps^*(s))' (N\kappa_N^{1/N}s^{1-1/N})^{\frac{p}{p-1}} \le \\
  \le \left[ \int_0^s \lambda \left(\frac{\kappa_N}{\varrho}\right)^{\gamma/N}
    \left\{
        \exp\left(
          \int_\varrho^{s} \frac{1}{{(N\kappa_N^{1/N})}^{p-q}}
\frac{ {(-(u_\eps^*)'(\tau))}^{q-(p-1)}}{\tau^{(1-1/N)(p-q)}}d\tau
        \right)
      \right\}d\varrho
    \right]
    \; \text{a.e. in} ]0,|\Omega|[.
\end{multline}

Given $t,h>0$, we take $\varphi=(T_{t+h}(u_\eps) - T_{t}(u_\eps))\sign
u_\eps$  as test function for \eqref{eq:gen}. Hence we get
\begin{equation}
  \label{eq:primo}
-\frac{d}{dt} \int_{|u_\eps|>t} H(Du_\eps)^p dx \le \int_{|u_\eps|>t}
H(D u_\eps)^q dx + \int_{|u_\eps|>t} \frac{\lambda}{H^o(x)^\gamma}.  
\end{equation}
The H\"older inequality gives that
\[
\int_{|u_\eps|>t} H(Du_\eps)^q dx \le \int_t^{+\infty}
\left[
  \left(
    -\frac{d}{d\tau} \int_{|u_\eps|>\tau} H(D u_\eps)^p dx
  \right)^{q/p}
  (-\mu'_{u_\eps}(\tau))^{1-q/p}
\right] d\tau.
\]
Hence, the H\"older inequality, the coarea formula \eqref{fr} and the
isoperimetric inequality \eqref{isop} give that
\[
\left( -\frac{d}{dt}
  \int_{|u_\eps|>t} H(Du_\eps)^p dx
\right)^{\frac{p-q}{p}}
  \ge (N \kappa_N^{1/N}\mu_{u_\eps}(t)^{1-1/N})^{p-q}
  (-\mu_{u_\eps}'(t))^{-\frac{(p-1)(p-q)}{p}}.
\]
Using the above inequalities and the Hardy-Littlewood inequality in
\eqref{eq:primo}, we obtain that
\begin{multline*}
  -\frac{d}{dt} \int_{|u_\eps|>t} H(Du_\eps)^p dx \le \\
  \le \int_0^{\mu_{u_\eps}(t)}
  \lambda\left(\frac{\kappa_N}{\varrho}\right)^{\gamma/N} d\varrho
  +\frac{1}{(N\kappa_N^{1/N})^{p-q}}\int_t^{+\infty}
  \left(
    -\frac{d}{d\tau} \int_{|u_\eps|>\tau} H(Du_\eps)^p dx
  \right)
  \left( \frac{-\mu_{u_\eps}'(\tau)}{(\mu_{u_\eps}(\tau))^{1-1/N}}
  \right)^{p-q} d\tau.
\end{multline*}
The Gronwall Lemma guarantees that
\begin{multline*}
  -\frac{d}{dt} \int_{|u_\eps|>t} H(Du_\eps)^p dx \le \int_0^{\mu_{u_\eps}(t)}
  \lambda\left(\frac{\kappa_N}{\varrho}\right)^{\gamma/N} d\varrho +
\int_t^{+\infty}   \frac{1}{(N\kappa_N^{1/N})^{p-q}}
  \left( \frac{-\mu_{u_\eps}'(\tau)}{(\mu_{u_\eps}(\tau))^{1-1/N}}
  \right)^{p-q} \times \\
  \times
  \left(
    \int_0^{\mu_{u_\eps}(\tau)}
    \lambda\left(\frac{\kappa_N}{\varrho}\right)^{\gamma/N} d\varrho
  \right)
  \exp\left\{
  \int_t^{\tau}   \frac{1}{(N\kappa_N^{1/N})^{p-q}}
  \left( \frac{-\mu_{u_\eps}'(r)}{(\mu_{u_\eps}(r))^{1-1/N}}
  \right)^{p-q}dr \right\} d\tau.
\end{multline*}
As matter of fact, reasoning as in \cite{bz88,tabest} (see also\cite{estro,fvpolya,fvpolya2}) it is possible to prove that
\begin{multline*}
\int_t^{\tau} \left( \frac{-\mu_{u_\eps}'(r)}{(\mu_{u_\eps}(r))^{1-1/N}}
  \right)^{p-q}dr = \frac{1}{N^{q-(p-1)}\kappa_N^{\frac{N-p+q}{N}}}
  \int_{\tau > u_\eps^\star(x)>t}
  \frac{H(Du_\eps^\star)^{q-(p-1)}}{H^o(x)^{N-1}} dx = \\
  =\int_{\mu_{u_\eps}(\tau)}^{\mu_{u_\eps}(t)}
  \frac{ {(-(u_\eps^*)'(r))}^{q-(p-1)}}{r^{(1-1/N)(p-q)}} dr.
\end{multline*}
Then we can proceed similarly than \cite{fermes}, and get
\eqref{eq:ineq}.

Now we observe that the solution $v$ obtained in Theorem \ref{unic2}
verifies \eqref{eq:ineq}, where the inequality is replaced by an
equality. Hence, from now on, recalling that the function $V(x)$
defined in \eqref{eq:expip} verifies \eqref{eq:expip1}, we can follow
line by line the proof of \cite[Theorem 4.1]{fermes}, in order to
get that
\begin{equation*}
  (-u_\eps^*(s))'\le (-v^*(s))', \quad\text{for a.e. }s\in ]0,|\Omega|],
\end{equation*}
and this gives the quoted comparison \eqref{eq:thesi}.
\end{proof}
From the proof of the above Theorem, we easily get estimates of the
solutions in Lebesgue and Sobolev spaces.
\begin{prop}
\label{propest}
  Under the assumptions of Theorem \ref{theo:comp}, the following
  uniform estimates hold:
  \begin{itemize}
  \item[(1)] if $p\ge q>\frac{N}{N-1}(p-1)$,
    \[
    \|u_\eps\|_s\le C, 
    \]
    for all $s<+\infty$ if $p=q$, or $s<\frac{N[q-(p-1)]}{p-q}$
    otherwise.
  \item[(2)] if $p\ge q>p-1+\frac p N$, then
    \[
    \| Du \|_p \le C.
    \]
  \item[(3)] if $p-1+\frac p N\ge q > \frac{N}{N-1}(p-1)$, then
    \[
    \| DT_k(u_\eps) \|_p \le C,\quad \| D u_\eps  \|_r \le C,
    \]
    for any $k>0$ and all $r<N[q-(p-1)]$.
  \end{itemize}
  In any case, $C$ denotes a constant independent on $\eps$.
\end{prop}
\begin{proof}
  Using \eqref{eq:thesi} and the equimeasurability of the
  rearrangements, we have that
  \[
  \|u\|_s \le \|v\|_s,
  \]
  and the explicit expression of $v$, given by
Theorem \ref{unic2}, allows to obtain immediately the estimate
  in (1).
  
  In order to get the gradient estimates in (2) and (3), we recall the
  proof of Theorem \ref{theo:comp}, and integrate by parts in
\eqref{eq:ineq}. It follows that 
\begin{multline}
  \label{eq:ineqgrad}
  -\frac{d}{dt} \int_{|u_\eps|>t} H(Du_\eps)^pdx \le \\ \le
  \lambda \int_t^{+\infty}
  \left(\frac{\kappa_N}{\mu_{u_\eps}(\tau)}\right)^{\gamma/N}
(-\mu_{u_\eps}'(\tau)) \exp 
  \left\{
    \frac{1}{(N\kappa_N^{1/N})^{p-q}}
    \int_t^{\tau}
    \left(
      \frac{-\mu'_{u_\eps}(r)}{\mu_{u_\eps}(r)^{1-1/N}}
    \right)^{p-q}dr 
  \right\}d\tau \le \\ \le
  \lambda \kappa_N^{\gamma/ N} \int_0^{\mu_{u_\eps}(t)}
  \varrho^{-\gamma/N}
  \exp\left\{
    \int_{({\varrho}/{\kappa_N})^{1/N}}^{({\mu_{u_\eps}(t)}/{\kappa_N})^{1/N}}
    (-v'(r))^{q-(p-1)}dr \right\}d\varrho.
\end{multline}
Last inequality follows by a change of variable and \eqref{eq:thesi},
recalling also that $v(r)=v^*(\kappa_N r^N)$. 

Hence, substituting the explicit expression of $v$, after some
computation we get that
\begin{equation}
  \label{eq:fin}
  -\frac{d}{dt} \int_{|u_\eps|>t} H(Du_\eps)^pdx \le
  C \mu_{u_\eps}^{1-\frac\gamma N}(t),
\end{equation}
where $C=C(N,\kappa_N,\gamma,\beta,\lambda)\ge 0$.

Now, suppose that $p\ge q>p-1-\frac{p}{N}$. Integrating
\eqref{eq:fin}, we get:
\begin{equation*}
  \int_\Omega H(Du_\eps)^p dx \le C \int_0^{|\Omega|}
  s^{^{1-\frac \gamma N}}(-u^*(s))'ds \le C \int_0^{|\Omega|}
  s^{-\frac{p}{N[q-(p-1)]} } ds, 
\end{equation*}
and the right-hand side is finite if and only if
$q>p-1+\frac{p}{N}$. This proves (2). 

Cosider now the condition in (3), $\frac{N}{N-1}(p-1)< q\le
p-1-\frac{p}{N}$. We have that
\[
-\frac{d}{dt} \int_{|u_\eps|>t} H(Du_\eps)^pdx =
\frac{d}{dt} \int_{|u_\eps|\le t} H(Du_\eps)^pdx \quad \text{a.e. in
}[0,+\infty[.
\]
Hence we can integrate \eqref{eq:fin} between $0$ and $k$ and
reason as before, obtaining that
\[
\int_\Omega |DT_k(u_\eps)|^p \le C k^{N-\frac{p}{q-(p-1)}}.
\]
Moreover, if $r<p$, using the H\"older inequality we get 
\begin{equation}
  \label{eq:hold}
-\frac{d}{dt} \int_{|u_\eps|> t} H(Du_\eps)^{r} dx \le
\left(
  -\frac{d}{dt} \int_{|u_\eps|> t} H(Du_\eps)^p dx
\right)^{\frac{r}{p}} [-\mu_{u_\eps}'(t)]^{1-\frac{r}{p}}.
\end{equation}
Using \eqref{eq:ineqgrad} and proceeding as before, we can integrate
both terms of \eqref{eq:hold}, obtaining that
\[
\int_{\Omega} H(Du_\eps)^r dx \le C \int_0^{|\Omega|}
s^{-\frac{r}{N[q-(p-1)]}} ds,
\]
which is finite if and only if $r<N(q-(p-1))$.
\end{proof}

\begin{proof}[Proof of Theorems \ref{thm:ex} and \ref{thm:comp}]
The estimates of Proposition \ref{propest}, in a standard way, allow to obtain that the approximating sequence $u_{\eps}$ converges, up to a subsequence, to a function $u$ which solves problem \eqref{eq:gen}. Moreover, $u^*_\eps\rightarrow u^*$ in some Lebesgue space, and then $u^{*}_{\eps}$ converges pointwise (up to a subsequence) to $u^*$ 
in $]0,|\Omega|]$. Passing to the limit in \eqref{eq:thesi}, we are done.
\end{proof}

\begin{rem}
 We stress that the bounds \eqref{eq:lin} and \eqref{eq:lin2} on $H$
  and $H^o$, and the conditions
  \eqref{ellipt}, \eqref{b=1} and \eqref{ipf} give that
  \[
	a(x,s,\xi)\cdot \xi \ge c_1^p |\xi|^p,\quad |b(x,s,\xi)| \le
  	c_2^q\,|\xi|^q,
  \]
  and
  \[
  f^\star(x) \le \frac{\lambda c_2^p}{|x|^p}.
  \]
  Hence, under the above growth conditions, the classical Schwarz
  symmetrization tecnique can be applied to problem \eqref{eq:gen}. In
  this way, it is possible to obtain results analogous to those of
  Theorems \ref{theo:comp} and Proposition \ref{propest}, and hence to those of 
  Theorems \ref{thm:ex}, \ref{thm:comp} (in the spirit of the existence results, for example,
  of \cite{fermes,afm14,femu13}), but requiring a stronger 
  assumption on the smallness of $\lambda>0$. This justifies the use of the more 
  general convex symmetrization (see also \cite{aflt} 
  and Remark 3.4 in \cite{dpg4}). 
\end{rem}
\begin{rem}
As regards the optimality of the smallness assumption on $f$, we refer the reader to \cite[Section 3]{afm14}. In such a paper the authors give some examples in the Euclidean radial case where, if $\lambda>0$ and \eqref{condf} is not satisfied, then in a suitable sense, there are no solutions.
\end{rem}

\bibliography{/Users/francescodellapietra/Documents/Biblioteca/library}{}
\bibliographystyle{abbrv}
\end{document}